\documentclass[12pt,oneside]{article}
\usepackage{amssymb,amsmath,amsthm,amscd,epsf,latexsym,verbatim,graphicx,amsfonts,hyperref,enumerate,caption}
  \input epsf.tex
\usepackage[usenames]{color} 

\theoremstyle{theorem}
\newtheorem{theorem}{Theorem}
\newtheorem{lemma}[theorem]{Lemma}
\newtheorem{proposition}[theorem]{Proposition}
\theoremstyle{definition}
\newtheorem{definition}[theorem]{Definition}

\newtheorem{corollary}[theorem]{Corollary}
\newtheorem{example}[theorem]{Example}

\newcommand{\R} {\mathbf{R}}

    \def\Int{\operatorname{Int}}   
        \def\pl{\partial}   

\title{Average pace and horizontal chords}
\author{Keith Burns, Orit Davidovich and Diana Davis}

\begin{document}
\maketitle


\section{Is there a mile at the average pace?}

On November 16, 2013,  Molly Huddle ran 37:49 for 12 kilometers, a world record for that distance.\footnote{Technically, Huddle's time was a \emph{world best}, since 12km is a non-standard distance.} People applauded this fine performance, but some pointed out that Mary Keitany's world record of 65:50 for the half marathon, which is 21.1 kilometers, is actually \emph{faster} than Huddle's record: Keitany averaged 3:07 per kilometer, while Huddle averaged 3:09 per kilometer \cite{iaaf} \cite{nyrr}. Therefore, Keitany must have run some 12 km subset of the race faster than Huddle $-$ right?

\begin{figure*}[!h] 
\centering
\includegraphics[height=120pt]{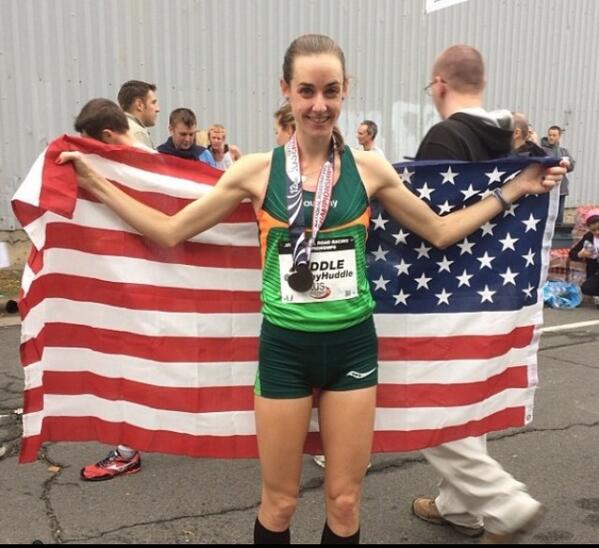} \ \ \ 
\includegraphics[height=120pt]{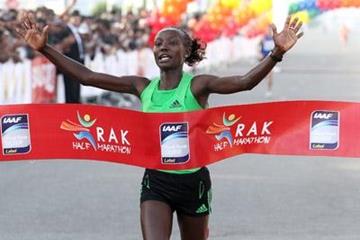}
\begin{quote}\caption*{\small{Molly Huddle (left) and Mary Keitany (right)   after their respective record-breaking runs. \textcopyright Victah Sailer, PhotoRun} } \end{quote}
\vspace{-2em}
\end{figure*}

No! Not necessarily:

\begin{example} \label{fsf}
Suppose that  Keitany ran 27:00 for the first and last 9.1 km, and 11:50 for the middle 2.9 km (Figure \ref{keitany}). Then her total time for the race would still be $2 \times \text{27:00} +  \text{11:50} = \text{65:50}$, but her time for each  12 km subinterval would have been $\text{27:00} + \text{11:50} = \text{38:50}$, much slower than Huddle's record.
\end{example}

\begin{figure}[!h] 
\centering
\vspace{-1em}
\includegraphics[width=300pt]{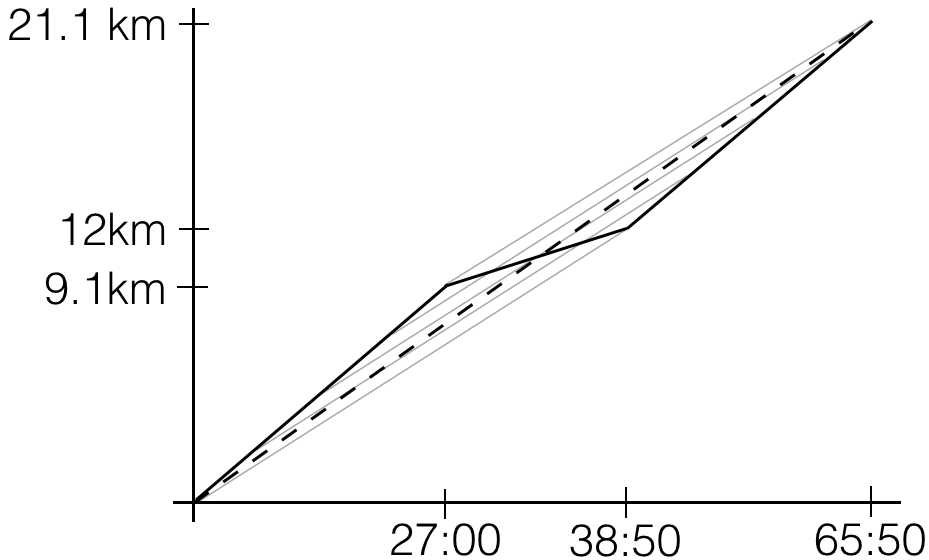}
\begin{quote}\caption{Distance covered as a function of time for a 21.1-km race run in 65:50, at the variable speed of  Example \ref{fsf} (solid), and at a constant speed (dashed). Scale is exaggerated to show the difference. 12km subsets of the race are shown in grey. \label{keitany}} \end{quote}
\vspace{-2em}
\end{figure}
\vspace{-1em}

Geometrically, we can see that every 12-km subinterval is covered in 38:50 by drawing a chord to the solid graph, with a vertical displacement of 12km. (A \emph{chord} is a line segment connecting two points on the graph.) Five examples are shown in grey in Figure \ref{keitany}. Every such chord  has a slope that is less than the slope of the dashed graph, so every 12km subset of the race is covered slower than the dashed line (Keitany's average speed). By construction, every such chord has a horizontal displacement of exactly 38:50.

Motivated by this example, we ask: When must it be true that there is a subset of a race covered in exactly the average speed of the entire race? The surprising answer is: almost never! 

In fact, there must be a subset covered in the average speed {if and only if} the length of the entire race is an integer multiple of the length of the subinterval of interest.
\footnote{For a similar discussion of this result and a related problem, see \cite{cs} and \cite{memory}.} 
This result shows, for example, that if you ran a 3-mile race at an average pace of 6:00 per mile, there must have been some mile that you ran in 6:00, but if you ran a 3.1-mile (5km) race at an average pace of 6:00, there need not have been any mile that you ran in exactly 6:00.

Here is an easy counterexample: If you run a 1.5-mile race in 15 minutes, averaging 10 minutes per mile, run the first and last half mile at a constant pace in 4 minutes and the middle half mile at a constant pace in 7 minutes. Then the first mile is covered in 11 minutes, the last mile is covered in 11 minutes, and sliding the endpoints of the mile we are looking at exchanges one fast part for another, so \emph{every} mile is covered in 11 minutes. 

Going the other way, if you run the first and last half mile in 6 minutes each, and the middle half mile in 3 minutes, your pace for every mile subset would be a speedy 9 minutes, but your average pace only 10 minutes per mile. See Figure \ref{smalls}(a) for position functions illustrating these possibilities.

Showing that an integer-length race has a mile at exactly the average speed is an application of the Intermediate Value Theorem:

\begin{proposition} \label{partition}
If the race distance is a whole number of miles, then some mile must be covered at exactly the average pace.
\end{proposition}

\begin{proof}
Let the total time for the race be $T$, and let the race distance be $n$ miles, with $n\in\mathbf{N}$. Partition the race into $n$  time subintervals of length $T/n$. 



If more than one mile is covered in every subinterval, then the total distance covered is more than $n$ miles, which is impossible. Similarly, if less than one mile is covered in every subinterval, then the total distance covered is less than $n$ miles, which is impossible. Thus there must be some subinterval in which at least a mile was covered, and some subinterval in which at most a mile was covered. Since the distance traveled by the runner in a time interval depends continuously on the start and end points of the interval,  
by the Intermediate Value Theorem there must be some intermediary $T/n$-length subinterval in which exactly  a mile was covered, establishing the result.\footnote{D.D. thanks Jon Chaika for a productive conversation about this result.}
\end{proof}

When I have told people about this problem, they usually suggest applying the Intermediate Value Theorem to the whole thing, with reasoning like: ``If you didn't run the race at a perfectly steady pace, then some part was faster than your average pace, and some part was slower, so by the Intermediate Value Theorem, you have to have a mile in between at the average pace.'' The problem with this reasoning is that, as we've seen in our examples, it's possible to arrange the fast parts and slow parts in such a way that \emph{every} mile (or 12km subset, or whatever your sub-interval of interest) is slower than the average pace. 

In the next section, we will show that it is possible to construct such a ``paradoxical race plan'' for \emph{any} non-integer distance.

\section{The Universal Chord Theorem}

The running problem is equivalent to an old and beautiful result called the \emph{Universal Chord Theorem}.  In the rest of the paper, we will show the equivalence of the two problems, explain the theorem, and explore the ideas of horizontal chord sets in general.

First, we'll translate the running problem to an equivalent problem about horizontal chords, which Paul L\'evy solved in 1934 \cite{levy}:


\begin{proposition}\label{hc}
For any positive non-integer $L>1$, there is a continuous function $f:L\to\mathbf{R}$ with $f(0)=f(L)=0$ whose graph has \emph{no} unit-length horizontal  chord.
\end{proposition}

The equivalence of the running problem with Proposition \ref{hc} is illustrated in Figure \ref{smalls}, and works as follows:


\begin{figure}[!h] 
\centering
\includegraphics[width=175pt]{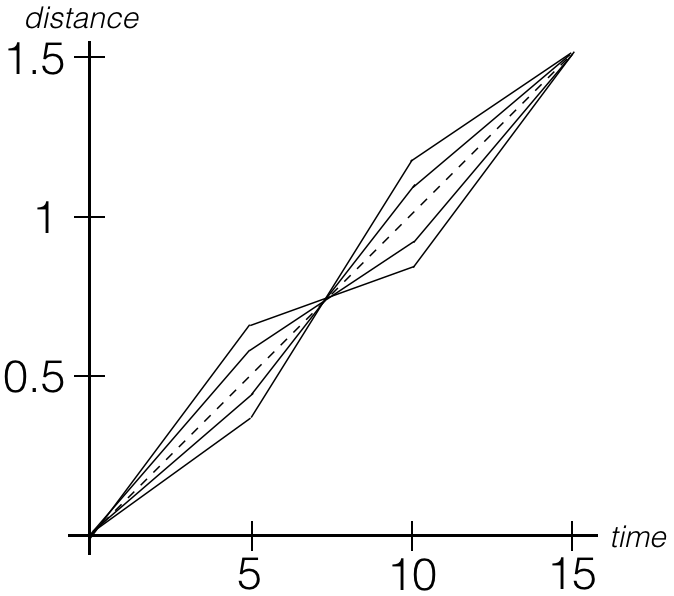} \ \ \ 
\includegraphics[width=175pt]{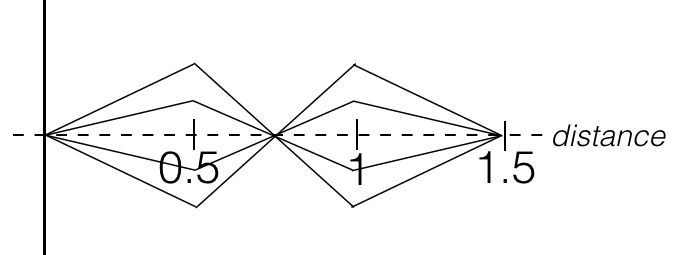}
\begin{quote}\caption{(a) Position functions for a 1.5-mile race averaging 10:00 pace with no 1-mile subset in exactly 10:00 (b) Functions from $[0,1.5]\to\mathbf{R}$ with no unit-length horizontal chord. The functions in (b) are vertical shears of those in (a). Each function is composed of $3$ segments. \label{smalls}} \end{quote}
\vspace{-2.5em}
\end{figure}

In the running problem, we want to construct a continuous position function for an $L$-mile race covered in $T$ minutes, so that no 1-mile subset of the race is covered in exactly T/L minutes. This means that we want a continuous position function \mbox{$f:[0,T]\to[0,L]$} such that $f(0)=0$ and $f(T)=L$, and no chord of the function simultaneously has a horizontal displacement of T/L  minutes and a vertical displacement of 1 mile. Figure \ref{smalls}(a) shows an example of solutions to this running problem for $L=1.5$ and $T=$ 15:00.

We can vertically shear the entire problem, so that we instead wish to find a continuous function with $f(0)=0$ and $f(T)=0$, and with no \emph{horizontal} chord of width T/L. Finally, we can re-scale the horizontal axis so that T/L is one unit. Then our task is to find a continuous function $f:[0,L]\to\mathbf{R}$ such that $f(0)=0$ and $f(L)=0$, with no unit-length horizontal chord, which is exactly the statement of Proposition \ref{hc}. All of these transformations are invertible, so the two constructions are equivalent. Figure \ref{smalls}(b) shows the equivalent solution to the function problem, which has $L=1.5$.


An example of a continuous function on $[0,L]$ with endpoints at $0$ and no unit-length horizontal chord is shown in Figure \ref{sawtooth}, with $L=4.4$. We graph $f(x)$ and $f(x-1)$, and carefully construct $f$ (thick) to avoid $f(x-1)$ (thin), so that $f$ has no unit horizontal chord.

\begin{figure}[!h] 
\centering
\includegraphics[width=250pt]{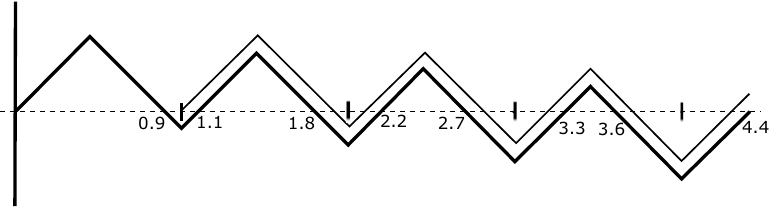}
\begin{quote}\caption{The thick function $f(x)$ has $f(0)=f(4.4)=0$. The thin function is $f(x-1)$. Since $f(x)$ and $f(x-1)$ do not intersect, $f$ has no unit horizontal chord. \label{sawtooth}} \end{quote}
\vspace{-2.5em}
\end{figure}

The function in Figure \ref{sawtooth} would also work for \mbox{$L=1.1, 1.8, 2.2, 2.7, 3.3, 3.6$} or {$4.4$}, all of the intersections with the $x$-axis.

The \textit{Universal Chord Theorem}, as it is known, has been repeatedly reworked over the past 200 years.  One statement reads (\cite{boas}, \cite{levit}):  
\begin{theorem}[Universal Chord Theorem]
For a given length $h$, the necessary and sufficient condition for a continuous function \mbox{$f:[0,L] \to\mathbf{R}$} with \mbox{$f(0)=f(L)$} to have a horizontal chord of length $h$ is that $h=L/m$ for some positive integer $m$.
\end{theorem}
 In his book \cite{boas}, Boas gave a history of the problem, tracing it back as early as Andr\'e-Marie Amp\`ere, who in 1806 proved the positive part of the assertion (see \cite{ampere}).  The modern history of the theorem begins with L\'evy, who in 1934 provided a complete proof (see \cite{levy}).  L\'evy's negative part of the assertion is proved using a counter-example function which is smooth:  For $h$ not an integer reciprocal of $L$, and a nontrivial smooth periodic function $\varphi(x)$, with period $h$, such that $\varphi(0)=0$, the function $f(x) = \varphi(x) - \frac{x}{L} \varphi(L) : [0,L] \rightarrow \mathbf{R}$ has $f(0) = 0 = f(L)$, and no horizontal chords of length $h$. In 1963, Levit 
showed that if $f$ is continuous on $[0,L]$ and changes sign $n$ times in that interval, and $f(0)=f(L)=0$, then $f$ has horizontal chords of every length between $0$ and $\frac{L}{\left \lfloor{(n+3)/2}\right \rfloor}$ (see \cite{levit}).



Assuming the Universal Chord Theorem, we can give a proof of the result equivalent to our running problem: 

\begin{proof}[Proof of Proposition \ref{hc}]
A continuous function $f:[0,L]\to\mathbf{R}$ \emph{must have} a horizontal chord of length $1$ if and only if $L$ is an integer. For a positive non-integer $L>1$, this condition is not satisfied, so such a function exists with no horizontal chord of length $1$.
\end{proof}

Given a candidate function $f(x)$, it is easy to check whether it has a horizontal chord of length $1$ or not: graph $f(x)$ and $f(x-1)$ on the same axes, and see whether they intersect, as in Figure \ref{sawtooth}. The more difficult task is to construct such a function in the first place! In the next section, we explore the work of Hopf, who solved an amazing generalization of the problem.

\section{Horizontal chord sets}\label{sec:hcs}

So far, we have seen functions that have  \emph{no unit-length} horizontal chord. In 1937, Heinz Hopf came along and completely solved the problem of what horizontal chords a continuous function can have \cite{hopf}. In the remainder of the paper, we will discuss this fascinating result.

\begin{definition}
For a function $f(x)$, its \emph{horizontal chord set} is the set of lengths horizontally connecting two points on the graph, i.e.
\[
S(f) = \{\ell\in\R: \text{ there exists } x\in\R \text{ with } f(x) = f(x+\ell)\}. 
\]
\end{definition}

Our running question above asked whether it is possible for a continuous function's horizontal chord set to exclude the number $1$. 
%
Hopf solved this problem in full generality:

\begin{theorem}[Hopf]\label{hopftheorem}
A given set $S \subset [0,\infty)$ is the horizontal chord set of some continuous function if and only if its complement $S^*$ is \emph{open} and \emph{additive}.
\end{theorem}

\begin{definition} \label{additivedef}
A set $X$ is \emph{additive} if $a,b\in X \implies a+b\in X$.
\end{definition}



Given any such set $S$ whose complement is open and additive, Hopf showed how to construct a continuous function $h_S$ whose horizontal chord set is exactly $S$. These functions look like our example in Figure \ref{sawtooth}. 

The implication that the horizontal chord set of a continuous function must be closed, and its complement in $(0,\infty)$ additive, is discussed in the Appendix. Here we will restrict attention to the opposite inference. More specifically, we will discuss the construction of a continuous function given a closed set $S \subseteq [0,\infty)$ whose complement is additive.

Since open additive sets are the key to understanding horizontal chord sets, you might wonder, what do open additive sets look like? 

\begin{example}\label{fourpointfour}
Let's look at one example, from the function in Figure \ref{sawtooth}. For this example, the horizontal chord set is
$$S=[0,0.9]\cup[1.1,1.8]\cup[2.2,2.7]\cup[3.3,3.6]\cup\{4.4\}.$$  
The complement of $S$ is the set
$$S^*=(0.9,1.1)\cup(1.8,2.2)\cup(2.7,3.3)\cup(3.6,4.4)\cup(4.4,\infty).$$ 

The reader will note, for example, that the interval $[0,0.9]$ covers all chord lengths of chords connecting endpoints lying on slopes belonging to a single `bump'. However, it also covers all chord lengths of chords connecting endpoints lying on the downward slopes of two consecutive `bumps'. The interval $[1.1,1.8]$, on the other hand, covers all chord lengths of chords connecting endpoints lying on the upward slopes of two consecutive `bumps', and so on it goes until we get to 4.4 $-$ the length of the longest chord connecting the outermost slopes.

One can check that $S^*$ is additive. 
 Notice that intervals of $S$ get shorter, while intervals of $S^*$ get longer. It turns out that this is always the case for additive sets that contain an interval. For example, our set $S^*$ contains the interval $(0.9,1.1)$. Since $S^*$ is additive, it must contain the double, triple, etc. of each point in this interval, so $S^*$ contains $(1.8,2.2), (2.7,3.3),$ $(3.6,4.4),(4.5,5.5),$ $(5.4,6.6)$ $-$ at which point our intervals have started to overlap, so $S^*$ contains a ray to infinity: Since $S^*$ contains $1$ (in the first interval) and an interval of length $1$ (from $4.5$ to $5.5$), and $S^*$ is additive, $S^*$ must contain every number above $4.5$. 
 \end{example}
 
In the same way we can see that no matter how small the first open interval contained in an open additive set is, the multiples of this set grow and eventually overlap, so an open additive set $S^*$ always contains a final interval of the form $(p,\infty)$. Additionally, this shows that its complement $S$ is bounded.

\section{Constructing the function}

Now we will show, for any set $S\subset[0,\infty)$ whose complement is open and additive, how to construct a continuous function $h_S$ whose horizontal chord set is $S$.

 \begin{definition} Consider a set $S\subset[0,\infty)$ whose complement is open and additive.
Let $L=\sup S$. (The discussion above explains why $S$ is bounded.)  For $x \in [0,L]$, define 
\[
a(x) = \sup \{ y \in \pl S: y \leq x\} \text{ \ \ \ and \ \ \ } b(x) = \inf \{ y \in \pl S: y \geq x\}
\]
\end{definition}

The functions $a(x)$ and $b(x)$ pick out the endpoints of the interval containing $x$. If $x \in \partial S$, we have $a(x) = x = b(x)$; otherwise, clearly \mbox{$a(x) < x < b(x)$}. If $x \in \Int S = S \setminus \partial S$, then $(a(x),b(x)) \subseteq \Int S$ is the maximal open interval in $\Int S$ containing $x$. Similarly, If $x \in S^*$, then $(a(x),b(x)) \subseteq S^*$ is the maximal open interval in $S^*$ containing $x$.


 \begin{definition} For $x$ as above, define $\alpha(x) = x - a(x)$ and $\beta(x) = b(x) - x$.  
\end{definition}

The functions $\alpha(x)$ and $\beta(x)$ give the distances from a given $x \in [0,L]$ to the boundaries of the maximal open interval, either in $S$ or in $S^*$, that contains it (with the exception of $x \in \partial S$, in which case $\alpha(x) = 0 = \beta(x)$).  
%
%
These functions $\alpha$ and $\beta$ are \emph{not} included in Hopf's original construction; they are an innovation of the current authors to streamline the proof.

Using our notation, Hopf's definition of $h_S(x)$ then becomes
\begin{equation}\label{h_S}
h_S(x) = \begin{cases}  \hphantom{-}0 &\quad\text{if $x \in \partial S$,}\\
                         \hphantom{-} \min(\alpha(x),\beta(x)) &\quad\text{if $x \in \Int  S$,}\\ 
                         -\min(\alpha(x),\beta(x)) &\quad\text{if $x \in S^*$.}
                         \end{cases}
\end{equation}
 
For the set $S$ in Example \ref{fourpointfour}, $h_S$ turns out to be exactly the function in Figure \ref{sawtooth}. Here is how it works: Let us color the intervals in $R^+$ based on whether they are contained in $S$ (black) or its complement $S^*$ (grey). Let us focus attention for a moment on the set $[0,0.9] \subseteq S$ in black. For every $x \in [0,0.9]$, one has $\alpha(x) = x$ and $\beta(x) = 0.9 - x$. As $x$ ranges between 0 and 0.45, $h_S(x) = x$, while as $x$ ranges between $0.45$ and $0.9, h_S(x) = 0.9 - x$. At the midpoint $x = 0.45$, the two lines of  $h_S(x)$ come together to make a triangle facing up over the interval $[0,0.9] \subseteq S$. The formula in (1) results in triangles facing up over each interval in $S$ and triangles facing down under each interval in $S^*$.


\noindent\includegraphics[width=\textwidth]{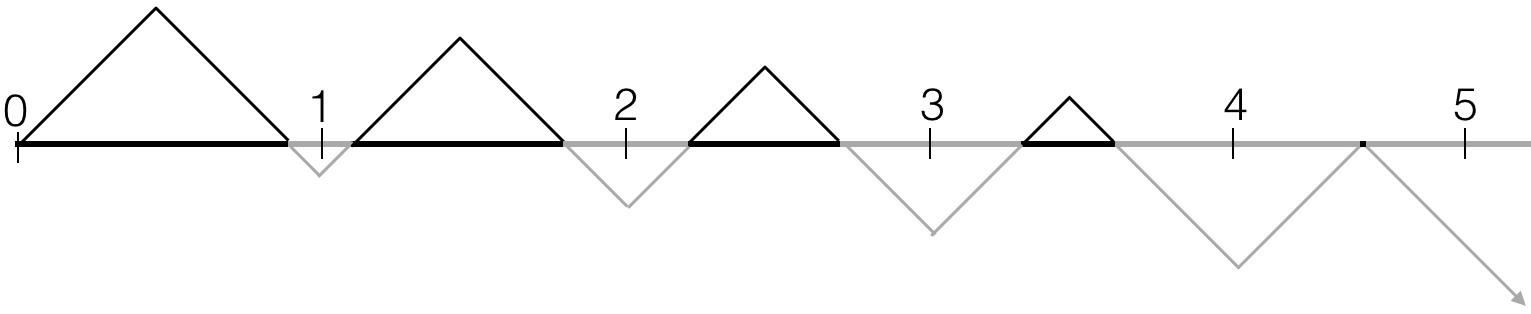}

Assuming Hopf's Theorem \ref{hopftheorem}, we can now give an alternative proof to Levy's Proposition \ref{hc}, which, as we discussed, is equivalent to our running problem:

\begin{proof}[Proof of Proposition \ref{hc}]
Recall that for any non-integer $L > 1$ we need to reproduce a continuous $f: [0,L] \rightarrow R$ with $f(0) = f(L) = 0$ such that $f$ has no unit-length chord. To put it in Hopf's terminology, $1 \not \in S(f)$ but $L \in S(f)$. Consider $\epsilon > 0$ smaller than $1/(n_L+1)$, the distance between $L$ and its nearest integer neighbor $n_L$. Consider the minimal open, additive $A \subseteq (0,\infty)$ that contains $(1-\epsilon,1+\epsilon)$, which would have to be \mbox{$\cup_{n=1}^{\infty} (n-n\epsilon,n+n\epsilon)$}. Consider the closed, bounded $S = [0,L] \setminus A$. Its complement, \mbox{$S^* = (0,\infty) \setminus S = A \cup (L,\infty)$}, is open and additive as well. Clearly, $1 \not \in S$ and $0 \in \partial S$. We are almost ready to invoke Hopf's construction. The function $f=h_S$ defined in the interval $[0,L]$ satisfies all we need. However, we should make sure $L \in S$ so that $0 = f(0) = f(L)$. However, our choice of $\epsilon$ was designed to make sure the two open intervals in $A$ containing the two integers nearest $L$ do not include $L$ itself.
%

Then $h_S(x)$ is a function $[0,L]\to\mathbf{R}$ with no unit horizontal chord.
\end{proof}

\section{Smoothing out the (sheared) position function}
 
There is no special reason why $h_S$ must be piecewise-linear. Variations on $h_S(x)$ would work equally well, for example $-h_S(x)$, or a function that connects successive $x$-intercepts with semicircles instead of triangles.   In fact, we can replace the function $\min(\alpha(x),\beta(x))$ in the definition \eqref{h_S} of $h_S$ with any function
$F(\alpha(x),\beta(x))$ that has the properties 
\begin{enumerate}
\item $F(\alpha,\beta) = 0$ if $\alpha = 0$ or $\beta=0$; \label{firstcondition}
\item $F(\alpha',\beta') > F(\alpha,\beta)$ if $\alpha' > \alpha>0$ and $\beta' > \beta>0$, \label{secondcondition}
\end{enumerate}
and all of the results and proofs still go through. Condition \eqref{firstcondition} makes sure that $h_S(x) = F(\alpha(x),\beta(x)) = 0$ for every $x \in \partial S$, while condition \eqref{secondcondition} implies that the `bumps' over the (finitely many) disconnected closed intervals that make up $S$ get smaller as the intervals themselves shrink in length to $0$.

John Oxtoby, in a survey article in the \emph{Monthly} in 1972, improved Hopf's result by constructing a $C^\infty$ function that does the same job as $h_S(x)$, but smoothly \mbox{(\cite{oxtoby}, Theorem 2)}. For our running application, we indeed might want our function to be smooth. Oxtoby's function is a variation of Hopf's function, and takes two pages to define, but we can give a ``smoother'' proof here using our definitions of $\alpha$ and $\beta$.

We create a smooth function with horizontal chord set $S$ as follows. We let $F(\alpha,\beta)= \phi(\alpha\beta)$, where $\phi$ is a $C^\infty$ function that is strictly increasing  on $(0,\infty)$ and has $\phi^{(n)}(0) = 0$ for all $n \geq 0$.  Then at a distance $x$ along an interval in $S$ or $S^*$ of length $a$, $\phi(\alpha\beta) = \phi(x(a-x))$.

So our smooth function with horizontal chord set $S$ is:

\begin{equation*}\label{f_S}
f_S(x) = \begin{cases}  \hphantom{-}0 &\quad\text{if $x \in \partial S$,}\\
                         \hphantom{-} \phi(\alpha(x),\beta(x)) &\quad\text{if $x \in \Int  S$,}\\ 
                         -\phi(\alpha(x),\beta(x)) &\quad\text{if $x \in S^*$.}
                         \end{cases}
\end{equation*}

\newpage

\section{A cautionary tale about continuity}

For all of the functions we discuss, we assume continuity. It turns out that if we don't assume that the function is continuous, \emph{any} set containing $0$ is the horizontal chord set of some function; openness and additivity of the complement are not required.

\begin{theorem}
Given $L\in\mathbf{R}^+$, suppose that $S\subset [0,L]$ and $0\in S$. Then there is a function $f:[0,L]\to\mathbf{R}$ whose horizontal chord set is $S$.
\end{theorem}

\begin{proof}
We define $f:[0,L]\to\mathbf{R}$ so that for every $y$ in the range of $f$, $f^{-1}(\{y\})$ contains either one or two elements. If it contains two elements, then those two elements will differ by some $s\in S$, thus ensuring that there is a horizontal chord of length $s$. If it contains one element, then it does not contribute to a horizontal chord and is essentially ``thrown away.''

We define $f$ by choosing its values one by one. If $L\in S$, then let \mbox{$f(0) = f(L) = 0$}. Now let $\lambda$ be the cardinality of $S\cap (0,L)$, and let $\{x_\alpha : \alpha < \lambda\}$ be an enumeration of the elements of $S\cap (0,L)$. By transfinite recursion, for each $\alpha < \lambda$, choose $x\in[0,L-x_\alpha]$ such that neither $f(x)$ nor $f(x+x_\alpha)$ has yet been defined, and choose a value $y>0$ that has not yet been used as a value of $f$, and let $f(x) = f(x+x_\alpha) = y$. At any stage in this construction, fewer than $2^{\aleph_0}$ values will have been used, so there will always be acceptable choices for $x$ and $y$. At the end of this process, if there are still points in $[0,L]$ where the value of $f$ has not yet been defined, map them to distinct negative numbers, for example mapping each remaining $x$ to $-x$.
\end{proof}

Thanks to an anonymous referee for pointing this out, and for giving the proof above.

\newpage

\section{Appendix: Proofs of Hopf's results}

In this section, we give Hopf's proofs of his result, because they are published in German, and are short. He states it as two different theorems, one for each direction of the implication given in Theorem \ref{hopftheorem}. Hopf proves the result not only for the graph of a function $f(x)$, but for a continuum:

\begin{definition}
A \emph{continuum} is a nonempty connected compact set. For a continuum $K$, let $S(K)$ be the set of lengths of horizontal chords in $K$: $s\in S(K)$ if and only if $s\geq 0$, and there are points $x$ and $x'$ in $K$ such that $x'-x=(s,0)$. Let $S^*(K)$ be the set of positive real numbers that are not in $S(K)$.
\end{definition}


Note that $S(K)$ is closed, so $S^*(K)$ is open.

\begin{theorem}[Hopf, Theorem I]\label{hopfthmi}
$S^*(K)$ is additive.
\end{theorem}



\begin{proof}
Let $K_s$ denote the set obtained by adding $(s,0)$ to all points of $K$; note that $K_0 = K$. With this notation, Theorem \ref{hopfthmi} says: If \mbox{$K_0 \cap K_a = \emptyset$} and \mbox{$K_0 \cap K_b= \emptyset$} (or equivalently \mbox{$K_a \cap K_{a+b} = \emptyset$}), then \mbox{$K _0\cap K_{a+b} = \emptyset$}. 

Let $K_s^\epsilon$ denote the $\epsilon$-neighborhood of $K_s$.
Since the sets  $K_s$ are all compact, we  can choose $\epsilon$ small enough so that \mbox{$K_0^\epsilon \cap K_a^\epsilon = \emptyset$} and \mbox{$K_a^\epsilon \cap K_{a+b}^\epsilon = \emptyset$} (Figure \ref{part1}).

\begin{figure}[!h] 
\centering
\includegraphics[width=\textwidth]{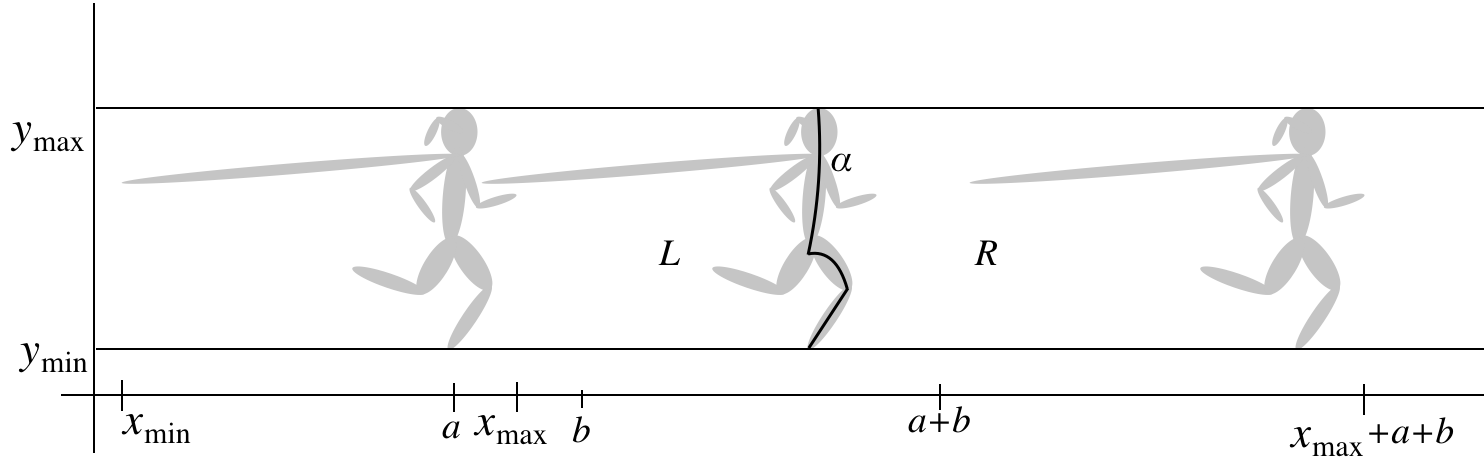}
\begin{quote}\caption{The construction in the proof of Theorem \ref{hopfthmi}. The shaded figures are from left to right $K^{\epsilon}_0, K^{\epsilon}_a, K^{\epsilon}_{a+b}$. \label{part1}} \end{quote}
\vspace{-2.5em}
\end{figure}

Let $y_{\text{min}}$ and $y_{\text{max}}$ be the minimum and maximum values of the $y$-coordinates of the  points in $K$. Similarly let $x_{\text{min}}$ and $x_{\text{max}}$ be the minimum and maximum values of the $x$-coordinates of the  points in $K$.    Since it is open and connected, $K_a^\epsilon$ contains a continuous curve with no self intersections that joins a point of $\Bbb R \times \{y_{\text{min}}\}$ to a point of $\Bbb R \times \{y_{\text{max}}\}$. This curve contains a subarc $\alpha$ with the same property that has no points outside $\Bbb R \times (y_{\text{min}},y_{\text{max}})$ other than its endpoints. The arc $\alpha$ divides the strip $\Bbb R \times [y_{\text{min}},y_{\text{max}}]$ into  a left half $L$ and a right half $R$ whose common boundary is~$\alpha$. The  points of $K_0$ at which the minimum $x$-coordinate $x_{\text{min}}$ is attained certainly lie in $L$ while the  points of $K_{a+b}$ at which the $x$-coordinate is $x_{\text{max}}+a+b$ lie in~$R$. But $K_0$ and $K_{a+b}$ cannot intersect $\alpha$. Thus all of $K_0$ lies in $L$ and all of $K_{a+b}$ lies in $R$. Therefore $K_0 \cap K_{a+b} = \emptyset$ because any intersection point would have to be in $\alpha$.
\end{proof}

\begin{corollary}
Let $S$ be the horizontal chord set of some continuous function \mbox{$f:[0,L]\to\mathbf{R}$}. Then its complement \mbox{$S^*=[0,\infty)\setminus S$} is open and additive.
\end{corollary}

\begin{theorem}[Hopf, Theorem II]\label{hopfthmii}
Let $S^*\subset (0,\infty)$ be nonempty, open and additive. 
Then its complement $S=[0,\infty)\setminus S^*$ is the horizontal chord set of some continuous function.
\end{theorem}

The function that Hopf constructs is $h_S(x)$, given in \eqref{h_S}. This function is defined on $[0,L]$, where $L$ is any number greater than sup $S$ (recall from \S \ref{sec:hcs} that $S$ is bounded).

Establishing the following properties of $S$ and $S^*$ is not difficult, so we omit the proofs. The reasoning is similar to the discussion of Example \ref{fourpointfour}; see \cite{hopf} for details.

\begin{lemma} \label{fivepartlemma} Let $S$ be a closed subset of $[0,\infty)$ whose complement $S^*$ is nonempty, open and additive. Then:
\begin{enumerate}[(a)]
\item $\overline{S^*}$ is additive.
\item $S$ is bounded. \label{bounded}
\item The infimum $l$ of $S^*$ is positive unless $S^* = (0,\infty).$\footnote{For instance, this is the case when $S$ is the chord set of a strictly increasing or strictly decreasing function.}
\item $S$ contains no interval  with  length $>l$. \label{lisbiggest}
\item If $s^* \in S^*$ and $s \in \pl S$, then $s+s^* \in S^*$. \label{boundary}
\end{enumerate}
\end{lemma}

 We can now establish three more properties of $S$ and $S^*$:
 
 \begin{lemma} \label{usinggreek}
 \begin{enumerate}[(a)]
 \item \label{alphabetains} For all $s\in S, \alpha(s)\in S$ and $\beta(s)\in S$.
 \item \label{insint}
Suppose $s^* \in S^*$ and both $x$ and $s^* + x$ are in $\Int S$. Then $$
 \alpha(s^* + x) < \alpha(x) \quad\text{and}\quad \beta(s^* + x) < \beta(x).
 $$
\item \label{insstar} Suppose $s^* \in S^*$ and both $x$ and $s^* + x$ are in $S^*$. Then  $$
\alpha(s^* + x) > \alpha(x) \quad\text{and}\quad \beta(s^* + x) > \beta(x).
 $$
 \end{enumerate}
 \end{lemma}
 
 \begin{proof} 
  \begin{enumerate}[(a)] 
  \item This follows from Lemma \ref{fivepartlemma} (\ref{lisbiggest}).
  
 By Lemma \ref{fivepartlemma} (\ref{boundary}), both $s^* + a(x)$ and $s^* + b(x)$ belong to $S^*$, which is open. From the definitions of $a(s^* + x)$ and  $b(s^* + x)$, we can obtain the inequalities
  \item   $s^* + a(x) < a(s^* + x) < b(s^* + x) < s^* + b(x)$, and
  \item $a(s^* + x) < s^* + a(x) <  s^* + b(x) <  b(s^* + x).$
 \end{enumerate}
 The results follow from each of these, respectively.
 \end{proof}

Now we can establish the desired results about the function $h_S(x)$:

\begin{proposition}
\begin{enumerate}[(a)]
\item The function $h_S$ has a  horizontal chord of length $s$ for each $s \in S$.
\item The length of a horizontal chord of the graph of $h_S$ lies in $S$.
\end{enumerate}
\end{proposition}

\begin{proof} (a) Note that $h_S: [0,L] \to \Bbb R$, defined in \eqref{h_S}, is a continuous function such that $h_S(x) \geq 0$ if $x \in S$ and $h_S(x) = 0$ if $x \in \pl S$. 
Since $b(s) \in \pl S$, we have \mbox{$h_S(b(s)) = 0$}. Translate the piece of the graph of $h_S$ joining $(s,h_S(s))$ to $(b(s),0)$ by $(-s,0)$. This produces an arc $A$ from $(0,h_S(s))$ to  $(\beta(s),0)$. Now observe that \mbox{$h_S(s) \geq 0 = h_S(0)$} and $0 \leq h_S(\beta(s))$, since $\beta(s) \in S$ by Lemma \ref{usinggreek} (\ref{alphabetains}). It follows that the arc $A$ crosses the graph of $h_S$ at a point $(x,h_S(x))$ for some $x \in [0,\beta(s)]$. More formally, we have
 $$
 h_S(0+s) = h_S(s) \geq 0 = h_S(0) \quad\text{and}\quad h_S(\beta(s) + s) = h_S(b(s)) = 0 \leq h_S(\beta(s)),
 $$
and it follows from the continuity of $h_S$ and the Intermediate Value Theorem that there is $x \in [0,\beta(s)]$ such that $h_S(x+s) = h_S(x)$. Since \mbox{$h_S(x) = h_S(x+s)$}, the points $(x,h_S(x))$ and $(x+s,h_S(x+s))$ are joined by a horizontal chord of length $s$.

 (b) It is evident that $h_S(x) > 0$ if $x \in \Int S$ and $h_S(x) < 0$ if $x \in S^*$. There are three cases to consider. 

\noindent {\bf Case 1:} The chord lies on the $x$-axis. In this case it must join two points $(s_1,0)$ and $(s_2,0)$ with both $s_1 \in \pl S$ and $s_2 \in \pl S$. By Lemma \ref{fivepartlemma} (\ref{boundary}), this is possible only if $|s_1- s_2| \notin S^*$.

\noindent {\bf Case 2:} The chord lies above the $x$-axis. In this case it must join two points $(s_1,y)$ and $(s_2,y)$ with both $s_1 \in \Int S$ and $s_2 \in \Int S$. We may assume that $s_1 \leq s_2$. It is immediate from Lemma \ref{usinggreek} (\ref{insint}) that if $s_2 - s_1 \in S^*$, then $\alpha(s_1) > \alpha(s_2)$ and $\beta(s_1) > \beta(s_2)$. But this is impossible if $h_S(s_1) = h_S(s_2)$.

\noindent {\bf Case 3:} The chord lies below the $x$-axis. In this case it must join two points $(s^*_1,y)$ and $(s^*_2,y)$ with both $s^*_1 \in S^*$ and $s^*_2 \in S^*$. We may assume that \mbox{$s^*_1 \leq s^*_2$}. It is immediate from Lemma \ref{usinggreek} (\ref{insstar}) that if $s^*_2 - s^*_1 \in S^*$, then \mbox{$\alpha(s^*_1) < \alpha(s^*_2)$} and $\beta(s^*_1) < \beta(s^*_2)$. But this is impossible if \mbox{$h_S(s^*_1) = h_S(s^*_2)$}.
\end{proof}

It is also interesting to note that the proof shows that $h_S(x+s^*) < h_S(x)$ whenever $s^* \in S^*$ and $0 \leq x < x+s^* \leq L$.

\end{document}